\documentclass{amsart}

\usepackage{mathptmx}
\usepackage{amsfonts,amsmath,amssymb,amsthm,amscd,amsxtra}
\usepackage{enumerate,verbatim}

\usepackage{amsfonts,amsmath,amssymb,amsthm,amscd,amsxtra}
\usepackage{enumerate,epsfig,verbatim}
\usepackage{tikz}
\usepackage{stmaryrd}
\usetikzlibrary{matrix,arrows}
\usepackage[framemethod=tikz]{mdframed}
\usepackage{tikz-cd}

\usepackage[usenames,dvipsnames]{pstricks}
\usepackage[mathscr]{eucal}
\usepackage{amsfonts,amsmath,amssymb,amsthm,amscd,amsxtra}
\usepackage{enumerate,verbatim}
\usepackage[all,2cell,ps]{xy}

\DeclareMathOperator{\Hdim}{\operatorname{\mathbf{\mathbb{I}}}}
\DeclareMathOperator{\rHdim}{\operatorname{\mathsf{red-\mathbb{I}}}}

\usepackage[pagebackref]{hyperref}
\usepackage{todonotes}
\usepackage{mathtools}
\usepackage{nicefrac}
\usepackage{array}
\usepackage{xcolor}

\usepackage[leqno]{amsmath}
\usepackage[pagebackref]{hyperref}

\usepackage[normalem]{ulem}

\usepackage{enumerate,verbatim}
\usepackage[all,2cell,ps]{xy}
\usepackage[notcite, notref]{}
\usepackage[pagebackref]{hyperref}
\usepackage{nicefrac}
\usepackage{array}
\usepackage{xcolor}
\usepackage{amsthm}
\usepackage[pagebackref]{hyperref}

\usepackage{amsfonts,amsmath,amssymb,amsthm,amscd,amsxtra}
\usepackage{enumerate,verbatim}
\usepackage{mathtools}

\usepackage[all,2cell,ps]{xy}
\usepackage{amsfonts}
\usepackage[margin=1.4in]{geometry}
\usepackage{color}
\usepackage[notcite,notref]{}
\usepackage{url}
\usepackage{datetime}
\usepackage{mathtools}

\usepackage[all,2cell,ps]{xy}
\usepackage{amssymb}
\usepackage{amsmath}
\usepackage{amsfonts}
\usepackage{amsmath}
\usepackage{amsthm}
\usepackage{amssymb}
\usepackage{amscd}
\usepackage{amsfonts}
\usepackage{amsxtra}     
\usepackage{epsfig}
\usepackage{verbatim}
\usepackage[all]{xypic}
\SelectTips{cm}{}

\def\latex/{{\protect\LaTeX}}
\def\latexe/{{\protect\LaTeXe}}
\def\amslatex/{{\protect\AmS-\protect\LaTeX}}
\def\tex/{{\protect\TeX}}
\def\amstex/{{\protect\AmS-\protect\TeX}}
\def\bibtex/{{Bib\protect\TeX}}
\def\makeindx/{\textit{MakeIndex}}

\usepackage{amsmath}
\usepackage{amsthm}
\usepackage{amssymb}
\usepackage{amscd}
\usepackage{amsxtra}     
\usepackage{epsfig}
\usepackage{verbatim}
\usepackage[all]{xypic}
\usepackage{enumerate}

\theoremstyle{plain} 
\newtheorem{thm}{Theorem}[section]

\newtheorem{prop}[thm]{Proposition}

\newtheorem{cor}[thm]{Corollary}

\theoremstyle{definition}
\newtheorem{chunk}[thm]{\hspace*{-1.065ex}\bf}
\newtheorem{lem}[thm]{Lemma}

\newtheorem{eg}[thm]{Example}

\newtheorem{rmk}[thm]{Remark}

\newcommand{\fm}{\mathfrak{m}}
\newcommand{\iso}{\cong}

\newcommand{\m}{\mathfrak{m}}

\newcommand{\NN}{\mathbb{N}}

\DeclareMathOperator{\id}{id}

\def\Tor{\operatorname{\mathsf{Tor}}}

\def\Ext{\operatorname{\mathsf{Ext}}}
\def\Hom{\operatorname{\mathsf{Hom}}}

\def\depth{\operatorname{\mathsf{depth}}}

\def\pd{\operatorname{\mathsf{pd}}}

\def\Tr{\operatorname{\mathsf{Tr}}}

\def\md{\operatorname{\mathsf{mod}}}

\DeclareMathOperator{\Gdim}{\mathsf{G-dim}}
\DeclareMathOperator{\rGdim}{\operatorname{\mathsf{red-G-dim}}}
\DeclareMathOperator{\rpdim}{\operatorname{\mathsf{red-pd}}}

 \DeclareMathOperator{\syz}{\Omega}

\def\u
rltilda{\kern -.15em\lower .7ex\hbox{\~{}}\kern
  .04em}\def\urldot{\kern -.10em.\kern -.10em}\def\urlhttp{http\kern
  -.10em\lower -.1ex\hbox{:}\kern -.12em\lower 0ex\hbox{/}\kern
  -.18em\lower 0ex\hbox{/}} 

\newcommand{\bb}{\left[ \begin{smallmatrix}}
\newcommand{\eb}{\end{smallmatrix} \right]}

\begin{document}
\baselineskip=15pt
\baselineskip=15pt
\title[On modules with finite reducing Gorenstein dimension]{On modules with finite reducing Gorenstein dimension}

\author{Tokuji Araya}
\address{Tokuji Araya \\ Department of Applied Science, Faculty of Science, Okayama University of Science, Ridaicho, Kitaku, Okayama 700-0005, Japan.}
\email{araya@das.ous.ac.jp}

\author{Olgur Celikbas}
\address{Olgur Celikbas\\ Department of Mathematics, West Virginia University, 
Morgantown, WV 26506 U.S.A}
\email{olgur.celikbas@math.wvu.edu}

\author{Jesse Cook}
\address{Jesse Cook \\ Department of Mathematics, West Virginia University, 
Morgantown, WV 26506 U.S.A}
\email{jcook27@mix.wvu.edu}

\author{Toshinori Kobayashi} 
\address{Toshinori Kobayashi \\ Graduate School of Mathematics, Nagoya University, Furocho, Chikusaku, Nagoya, Aichi 464-8602, Japan}
\email{m16021z@math.nagoya-u.ac.jp}

\thanks{2010 {\em Mathematics Subject Classification.} Primary 13D07; Secondary 13C13, 13C14, 13H10}
\keywords{Gorenstein dimension, reducing dimensions, injective dimension}
\thanks{Toshinori Kobayashi was partly supported by JSPS Grant-in-Aid for JSPS Fellows 18J20660}

\begin{abstract} If $M$ is a nonzero finitely generated module over a commutative Noetherian local ring $R$ such that $M$ has finite injective dimension and finite Gorenstein dimension, then it follows from a result of Holm that $M$ has finite projective dimension, and hence a result of Foxby implies that $R$ is Gorenstein. We investigate whether the same conclusion holds for nonzero finitely generated modules that have finite injective dimension and finite reducing Gorenstein dimension, where the reducing Gorenstein dimension is a finer invariant than the classical Gorenstein dimension, in general.
\end{abstract}

\maketitle

\section{Introduction}

Throughout $R$ denotes a commutative Noetherian local ring with unique maximal ideal $\fm$ and residue field $k$, and modules over $R$ are assumed to be finitely generated.

In 1969 Levin and Vascensolos \cite[2.2]{LV} proved that, if $R$ is a Gorenstein ring, then an $R$-module has finite projective dimension if and only if it has finite injective dimension. Subsequently, in 1977 Foxby \cite[4.4]{Foxby77} proved a surprising converse: if $R$ admits a nonzero module of finite projective and finite injective dimension, then $R$ must be Gorenstein. Nearly three decades later, Holm \cite[2.2]{Holm} improved Foxby's result by considering modules (not necessarily finitely generated) of finite Gorenstein projective dimension. Holm's result \cite{Holm}, in the local case, implies that, if $M$ is an $R$-module of finite injective dimension, then the projective dimension of $M$ equals the Gorenstein dimension of $M$; see \ref{Gdim}. In the local setting, the results of Foxby and Holm from the foregoing discussion are summarized as: 

\begin{thm} \label{Holm} (Foxby \cite[4.4]{Foxby77} and Holm \cite[2.2]{Holm}) Let $R$ be a local ring and let $M$ be a nonzero $R$-module such that $\id_R(M)<\infty$. Then the following hold:
\begin{enumerate}[\rm(i)]
\item $\Gdim_R(M)=\pd_R(M)$. 
\item If $\Gdim_R(M)<\infty$, then $R$ is Gorenstein.
\end{enumerate}
\end{thm}

Araya and Celikbas \cite{CA}, motivated by the work of Alperin \cite{Alp}, Avramov \cite{Av1}, and Bergh \cite{Be, Berghredcx2} on the complexity of modules, introduced and studied the notion of reducing homological dimensions. These homological dimensions have been recently considered in the noncommutative setting by Araya and Takahashi \cite{AT}. In general, a module may have infinite, but finite reducing, homological dimension; see \ref{rGdim} and Example \ref{exCA} for the details.

The main purpose of this paper is to improve Theorem \ref{Holm}: we investigate whether the conclusion of the theorem holds when the Gorenstein dimension and the projective dimension are replaced with their reducing versions. We prove that the conclusion of the first part of Theorem \ref{Holm} also holds for reducing homological dimensions. Moreover, we are able to extend the conclusion of the second part of the theorem for two distinct classes of modules. More precisely, we prove:

\begin{thm} \label{mainthm} Let $R$ be a $d$-dimensional local ring. If $M$ is a nonzero $R$-module such that $\id_R(M)<\infty$, then the following hold:
\begin{enumerate}[\rm(i)]
\item $\rGdim_R(M)=\rpdim_R(M)$.
\item If $\rGdim_R(M) \leq 1$, or $\rGdim_R(M)<\infty$ and $\depth_R(M)\geq d-1$, then $R$ is Gorenstein.
\end{enumerate}
\end{thm}

One of the motivations for Theorem \ref{mainthm} comes from a result of Araya and Celikbas, which establishes Theorem \ref{mainthm} for the case where $M$ is maximal Cohen-Macaulay; see \ref{rGdimMCM}.
A nontrivial consequence of Theorem \ref{mainthm} is that, if $R$ is a one-dimensional local ring and $M$ is a nonzero $R$-module such that $\id_R(M)<\infty$ and $\rGdim_R(M)<\infty$, then $R$ is Gorenstein. Furthermore, for the two-dimensional case, it follows immediately from Theorem \ref{mainthm} that:

\begin{cor} Let $R$ be a two-dimensional local ring and let $M$ be a nonzero torsion-free $R$-module (e.g., $M$ is an ideal of $R$). If $\id_R(M)<\infty$ and $\rGdim_R(M)<\infty$, then $R$ is Gorenstein.
\end{cor}

The proof of Theorem \ref{mainthm} requires preparations; in section 2 we record several preliminary results and give a proof of Theorem \ref{mainthm}. Section 3 is devoted to the proofs of the preliminary results used to establish Theorem \ref{mainthm}.

\section{preliminaries and the proof of Theorem \ref{mainthm}}

In this section we prove our main result, namely Theorem \ref{mainthm}.  Along the way, we record several preliminary results that are used in the proof and an argument for Theorem \ref{Holm}.


\begin{chunk} \label{Gdim} (\textbf{Gorenstein dimension} \cite{Tome1, AuBr}) Let $R$ be a local ring and let $M$ be an $R$-module. Then $M$ is said to be \emph{totally reflexive} provided that $M\cong M^{\ast\ast}$ and $\Ext^i_R(M,R) = 0 = \Ext^i_R(M^{\ast},R)$ for all $i\geq 1$. 

The infimum of $n$ for which there exists an exact sequence $0\to X_n \to \cdots \to X_0 \to M \to 0$ such that each $X_i$ is totally reflexive is called the \emph{Gorenstein dimension} of $M$. If $M$ has Gorenstein dimension $n$, we write $\Gdim_R(M) = n$. Therefore, $M$ is totally reflexive if and only if $\Gdim_R(M)\leq 0$, where it follows by convention that $\Gdim_R(0)=-\infty$.

In the proof of Theorem \ref{mainthm}, we use the fact that the category of modules of finite Gorenstein dimension is closed under taking direct summands; see, for example, \cite[1.1.10(c)]{Gdimbook}.\pushQED{\qed} 
\qedhere
\popQED	
\end{chunk}


Theorem \ref{Holm}(i) is one of the motivations for our work. Therefore, it seems worth recording a short argument that establishes the result; the argument we give here is different from that of Holm \cite{Holm} and  has been explained to us by Arash Sadeghi. First we recall:


\begin{chunk} \label{Isc} (Ischebeck \cite[2.6]{Ischebeck}) Let $R$ be a local ring and let $M$ and $N$ be nonzero $R$-modules.
If $\id_R(N)<\infty$, then it follows that $\depth(R)-\depth_R(M)=\sup\{n\mid \Ext^n_R(M,N)\not=0\}$. \pushQED{\qed} 
\qedhere
\popQED	
\end{chunk}


\begin{chunk} (Holm \cite[2.2]{Holm}) \label{Holm's} Let $R$ be a local ring and let $M$ be an $R$-module. If $\id_R(M)<\infty$, then it follows that $\pd_R(M)=\Gdim_R(M)$.

To see this, note that there is nothing to prove if $\Gdim_R(M)=\infty$. So we assume $\Gdim_R(M)=n<\infty$. Then it follows that $\Gdim_R(\Tr \Omega^n M)=0=\Gdim_R(\Tr \Omega^{n+1} \Tr \Omega^n M)$, where $\Tr$ denotes the Auslander transpose \cite{AuBr}. Next we consider the exact sequence that follows from \cite[2.8]{AuBr}:
\begin{equation}\tag{\ref{Holm's}.1}
0\rightarrow\Ext^1_R(\Tr \Omega^{n+1} \Tr \Omega^n M,M) \rightarrow \Tor_{n+1}^R(\Tr \Omega^n M,M)\rightarrow
\Hom_R(\Ext^{n+1}_R(\Tr \Omega^n M,R),M).
\end{equation}
As $\id_R(M)<\infty$, it follows from \ref{Isc} that $\Ext^1_R(\Tr \Omega^{n+1} \Tr \Omega^n M,M)=0$.
Moreover, we have that $\Ext^{n+1}_R(\Tr \Omega^n M,R)=0$. Thus (\ref{Holm's}.1) yields that $\Tor_1^R(\Tr \Omega^n M,\Omega^nM)=0$.
Consequently, we conclude that $\Omega^n M$ is free, i.e., $\pd_R(M)<\infty$; see \cite[3.9]{Yo}. Hence $\pd_R(M)=\Gdim_R(M)$, as required. \pushQED{\qed} 
\qedhere
\popQED	
\end{chunk}


\begin{chunk} \label{rGdim} (\textbf{Reducing dimensions} \cite[2.1]{CA}) Let $M$ be an $R$-module, and let $\Hdim$ be a \emph{homological invariant} of $R$-modules, for example, $\Hdim=\pd$ or $\Hdim=\Gdim$.

We write $\rHdim(M)<\infty$ provided that there exists a sequence of $R$-modules $K_0,  \ldots, K_r$, positive integers $a_1, \dots, a_r,b_1, \dots, b_r ,n_1, \dots, n_r$, and short exact sequences of the form 
\begin{equation}\tag{\ref{rGdim}.1}
0 \to K_{i-1}^{{\oplus a_i}} \to K_{i} \to \Omega^{n_i}K_{i-1}^{{\oplus b_i}} \to 0
\end{equation}
for each $i=1, \ldots r$, where $K_0=M$ and $\Hdim(K_r)<\infty$. If a sequence of modules as in (\ref{rGdim}.1) exists, then we call $\{K_0,  \ldots, K_r\}$ a \emph{reducing $\Hdim$-sequence} of $M$. 

The \emph{reducing invariant} $\Hdim$ of $M$ is defined as follows:
\begin{equation}\notag{}
\rHdim(M)=\inf\{ r\in \NN \cup \{0\}: \text{there is a reducing $\Hdim$-sequence }  K_0,  \ldots, K_r \text{ of }  M\}.
\end{equation}
We set, $\rHdim(M)=0$ if and only if $\Hdim(M)<\infty$. \pushQED{\qed} 
\qedhere
\popQED	
\end{chunk}


In passing we recall an example which shows that the reducing homological dimensions are finer than regular homological dimensions. This example also shows that the reducing homological dimension of a module is not always bounded by the depth of the ring in question.

\begin{eg} \label{exCA} (\cite[2.3]{CA}) Let $R=k[x,y]/(x,y)^2$. Then $\pd_R(k)=\infty=\Gdim_R(k)$, but we have that $\rGdim_R(k)=1=\rpdim_R(k)$. Also, if $M$ is an $R$-module such that $\rGdim_R(M)\leq \infty=\Gdim_R(M)$, then $M\cong R^{\oplus \alpha} \oplus k^{\oplus \beta}$ for some $\alpha\geq 0$ and $\beta \geq 1$.\pushQED{\qed} 
\qedhere
\popQED	
\end{eg}


The proof of Theorem \ref{mainthm} relies upon the following preliminary results. 

\begin{chunk} (Araya and Celikbas \cite[3.3(iii)]{CA}) \label{rGdimMCM} Let $R$ be a local ring and let $M$ be a nonzero $R$-module such that $\rGdim_R(M)<\infty$. If $\Ext^i_R(M,M)=0$ for all $i\geq 1$, then $\Gdim_R(M)<\infty$. Therefore, if $M$ is maximal Cohen-Macaulay and $\id_R(M)<\infty$, then $R$ is Gorenstein. \pushQED{\qed} 
\qedhere
\popQED	
\end{chunk}

\begin{prop} \label{app-prop} Let $R$ be a local ring and let $M$ be an $R$-module. Assume, whenever $X$ is a totally reflexive $R$-module, one has $\Ext^i_R(X, M)=0$ for all $i\geq 1$. Then it follows that $\rGdim_R(M)=\rpdim_R(M)$.
\end{prop}

\begin{prop} \label{p1} Let $R$ be a local ring and let $0\to M^{\oplus a} \to K \to \syz^{n} M^{\oplus b} \to 0$ be a short exact sequence of $R$-modules, where $a\geq 1$,  $b\geq 1$, and $n\geq 0$ are integers. If $\Gdim_R(K)<\infty$, then, for each $i\geq 1$, there exists a short exact sequence of $R$-modules $0\to M^{\oplus a_i} \to Y_i \to \syz^{r_i} M^{\oplus b_i} \to 0$, where $\Gdim_R(Y_i)<\infty$, $r_i=2^i(n+1)-1$, $a_i=a^{2^{i}}$ and $b_i=b^{2^{i}}$.
\end{prop}

\begin{prop} \label{mainprop} Let $R$ be a local ring and let $M$ be an $R$-module. If $x \in \m$ is a non zero-divisor on $R$ and $M$, then it follows that $\rGdim_{R/xR}(M/xM)\leq \rGdim_R(M)$.
\end{prop}

Now we make use of Propositions \ref{app-prop}, \ref{p1} and \ref{mainprop}, and prove Theorem \ref{mainthm}; we defer the proofs of these propositions to Section 3. 

\begin{proof}[Proof of Theorem \ref{mainthm}] Note that, since $\id_R(M)<\infty$, it follows that $\Ext^j_R(X, M)=0$ for all $j\geq 1$ for each totally reflexive $R$-module $X$; see \ref{Isc}.
Hence part (i) follows from Proposition \ref{app-prop}.

Next we assume $\rGdim_R(M)\leq 1$, and show that $R$ is Gorenstein. It follows from \ref{rGdim} that there exists a short exact sequence of $R$-modules $0\to M^{\oplus a} \to K \to \syz^{n} M^{\oplus b} \to 0$, where $a$, $b$, $n$ are positive integers and $\Gdim_R(K)<\infty$. Then, by Proposition \ref{p1}, we have a short exact sequence of $R$-modules $0\to M^{\oplus a_i} \to Y \to \syz^{r_i} M^{\oplus b_i} \to 0$, where $i\gg 0$, $\Gdim_R(Y)<\infty$, and $r_i\geq d$. Therefore, $\syz^{r_i} M^{\oplus b_i}$ is maximal Cohen-Macaulay. Now, as $\id_R(M^{\oplus a_i})<\infty$, we see that $\Ext^1_R(\syz^{r_i} M^{\oplus b_i} ,M^{\oplus a_i})=0$; see \ref{Isc}. So the short exact sequence $0\to M^{\oplus a_i} \to Y \to \syz^{r_i} M^{\oplus b_i} \to 0$ splits, and hence $M^{\oplus a_i}$ occurs as a direct summand of 
$Y$. This shows that $\Gdim_R(M)<\infty$ and hence $R$ is Gorenstein; \ref{Gdim} and Theorem \ref{Holm}.

Now, to complete the proof of part (ii), we assume $\depth_R(M)\geq d-1$, and proceed by induction on $d$ to show that $R$ is Gorenstein. There is nothing to prove if $d=0$; see \ref{rGdimMCM}. 

Assume $d\geq 2$ and that the claim is true when $d=1$. Since $R$ is a Cohen-Macaulay ring and $\depth_R(M)\geq 1$, there exists an element $x\in \fm$ which is a non zero-divisor on both $R$ and $M$. Then it follows by Proposition \ref{mainprop} that $\rGdim_{R/xR}(M/xM)\leq \rGdim_R(M)<\infty$. Hence, as $\id_{R/xR}(M/xM)<\infty$ and $\depth_{R/xR}(M/xM)\geq d-2$, we conclude by the induction hypothesis that $R/xR$ is Gorenstein, i.e., $R$ is Gorenstein. Therefore, it suffices to prove the case where $d=1$.
 
Assume $d=1$. Then the case where $\depth_R(M)=1$ follows from \ref{rGdimMCM}. Therefore, we assume $\depth_R(M)=0$ and choose a reducing $\Gdim$-sequence $\{K_0, \ldots, K_r\}$ of $M$. 

Claim: For each $i=0, \ldots, r$, we have that $K_i \cong M^{\oplus c_i} \oplus L_i$ for some $R$-module $L_i$ and for some integer $c_i\geq 1$ such that $L_i$ is either zero or maximal Cohen-Macaulay. 

Proof of the claim: We proceed by the induction on $i$. If $i=0$, then, since $K_0=M$, we pick $L_0=0$ and $c_0=1$. So we assume $i\geq 1$. Then, by the  induction hypothesis, we have that $K_{i-1}  \cong M^{\oplus c_{i-1}} \oplus L_{i-1}$ for some $R$-module $L_{i-1}$ and for some integer $c_{i-1}\geq 1$, where $L_{i-1}$ is either zero or maximal Cohen-Macaulay. Now we consider the following pushout diagram, where the middle horizontal short exact sequence follows by the definition of reducing Gorenstein dimension; see \ref{rGdim}.  
\[
\xymatrix{
& 0 \ar[d] & 0 \ar[d] & & \\
& M^{\oplus c_{i-1}a_i} \ar[d] \ar@{=}[r] & M^{\oplus c_{i-1}a_i} \ar[d]  & & \\
0 \ar[r] & K_{i-1}^{\oplus a_i} \ar[d] \ar[r] & K_i \ar[d] \ar[r] & \syz^{n_i}K_{i-1}^{\oplus b_i } \ar@{=}[d] \ar[r] & 0 \\
0 \ar[r] & L_{i-1}^{\oplus a_i} \ar[d] \ar[r] & L_i \ar[d] \ar[r] & \syz^{n_i} K_{i-1}^{\oplus b_i } \ar[r] & 0 \\
& 0 & 0 & &}
\]
Since $L_{i-1}$ is either zero or maximal Cohen-Macaulay, we see from the bottom horizontal short exact sequence that $L_i$ is either zero or maximal Cohen-Macaulay.
In either case, since $\id_R(M)<\infty$, it follows by \ref{Isc} that $\Ext^1_R(L_i, M)=0$. This implies that the middle vertical short exact sequence splits, yields the isomorphism $K_i \cong L_i \oplus M^{\oplus c_{i-1}a_i}$, and proves the claim.

Now, by the claim established above, $M$ is a direct summand of $K_r$. Then, since $\Gdim_R(K_r)<\infty$, we conclude that $\Gdim_R(M)<\infty$; see \ref{Gdim}. Therefore, Theorem \ref{Holm} shows that $R$ is Gorenstein, and this completes the proof of the theorem.
\end{proof}


\section{Proofs of the preliminary propositions}

This section is devoted to the proofs of Propositions \ref{app-prop}, \ref{p1} and \ref{mainprop}. We start by preparing a lemma:


\begin{lem} \label{app-lem} Let $R$ be a local ring and let $M$ be an $R$-module. Assume, for each totally reflexive $R$-module $X$, it follows that $\Ext^j_R(X, M)=0$ for all $j\geq 1$. Assume further there are short exact sequences of $R$-modules of the form $0 \to K_{i-1}^{{\oplus a_i}} \to K_{i} \to \Omega^{n_i}K_{i-1}^{{\oplus b_i}} \to 0$, where $r$ is a positive integer, $K_0=M$, and $a_1, \dots, a_r,b_1, \dots, b_r ,n_1, \dots, n_r$ are positive integers for each $i=1, \ldots r$. Then, for each totally reflexive $R$-module $X$, it follows that $\Ext^j_R(X, K_i)=0$ for all $j\geq 1$ and for all $i=0, \ldots, r$.
\end{lem}

\begin{proof} We proceed by induction on $i$. 

If $i=0$, then $K_0=M$, and so there is nothing to prove. Let $i\geq 1$ be an integer and assume, for each totally reflexive $R$-module $X$, we have that $\Ext^j_R(X, K_{i-1})=0$ for all $j\geq 1$.  

Next consider the following short exact sequence that exists by our assumption:
\begin{equation}\tag{\ref{app-lem}.1}
0 \to K_{i-1}^{{\oplus a_i}} \to K_{i} \to \Omega^{n_i}K_{i-1}^{{\oplus b_i}} \to 0.
\end{equation}
Let $Y$ be a totally reflexive $R$-module. Then (\ref{app-lem}.1) yields the exact sequence for each $j\geq 1$:
\begin{equation}\tag{\ref{app-lem}.2}
\Ext_R^{j}(Y, K_{i-1}^{{\oplus a_i}}) \to \Ext_R^{j}(Y, K_{i}) \to \Ext_R^{j}(Y, \Omega^{n_i}K_{i-1}^{{\oplus b_i}}). 
\end{equation}
As $\Ext_R^{j}(Y, K_{i-1}^{{\oplus a_i}})$ vanishes for each $j\geq 1$ by the induction hypothesis, to complete the induction argument, it suffices to observe the vanishing of $\Ext_R^{j}(Y, \Omega^{n_i}K_{i-1})$.
This follows due to the facts that $\Omega^{-n_i}Y$ is a totally reflexive $R$-module and $\Ext_R^{j}(Y, \Omega^{n_i}K_{i-1}) \cong \Ext_R^{j}(\Omega^{-n_i}Y, K_{i-1})$. 
\end{proof}

\begin{proof}[Proof of Proposition \ref{app-prop}:] \label{pp1} Note that, if $\rGdim_R(M)=\infty$, then it follows that $\rpdim_R(M)=\infty$. Hence, to prove the proposition, it suffices to assume $\rGdim_R(M)<\infty$.

Assume $\rGdim_R(M)=r<\infty$ and let $\{K_0, \ldots, K_r \}$ be a reducing $\Gdim$ sequence of $M$. Then, since $\Gdim_R(K_r)<\infty$, we consider the finite projective hull of $K_r$ \cite[1.1]{AuBu}, i.e., a short exact sequence of $R$-modules of the form $0 \to K_r \to P \to X \to 0$, where $\pd_R(P)<\infty$ and $X$ is totally reflexive. Note that Lemma \ref{app-lem} implies that $\Ext^1_R(X, K_r)=0$. Therefore, the finite projective hull of $K_r$ splits and hence $\pd_R(K_r)<\infty$. This shows that $\rpdim_R(M)\leq r$. As, in general, $\rGdim_R(M)\leq \rpdim_R(M)$, the claim of the proposition follows.
\end{proof}

\begin{rmk} It is worth noting that there are examples of local rings $R$ and modules $M$ over $R$ with $\rGdim_R(M)<\infty=\pd_R(M)$ and $\Ext^i_R(X, M)=0$ for each totally reflexive $R$-module $X$ and each $i\geq 1$. For example, if $R$ is as in Example \ref{exCA} and $M=k$, then each totally reflexive $R$-module is free so that $\Ext^i_R(X, M)=0$ for all $i\geq 1$ and $\pd_R(M)=\infty$. 
\end{rmk}
\vspace{-0.1in}

The next two results are used for the proof of Proposition \ref{p1}; the first one, \ref{R1}, is well-known, but we include it for completeness. The second one, \ref{Tak}, is a special case of \cite[3.1]{Taksyz} and plays an important role for the proof of Proposition \ref{p1}.

\begin{chunk} \label{R1} (\cite[2.2]{DaoTak}) Let $R$ be a local ring and let $0\to A\to B\to C \to 0$ be a short exact sequence of $R$-modules. Then there are exact sequence of $R$-modules, where $n\geq 0$ is any integer, and $F$ and $G$ are free $R$-modules.
\begin{enumerate}[\rm(i)]
\item $0 \to \syz C \to A\oplus F \to B \to 0$.
\item $0 \to \syz^n A \to G \oplus \syz^n B \to \syz^n C \to 0$.
\end{enumerate}
\end{chunk}

\begin{chunk} \label{Tak} Let $R$ be a commutative ring. If $0 \to L \to X \to N \to 0$ is a short exact sequence of $R$-modules such that $\Gdim_R(X)<\infty$ and $L\cong G \oplus Y$ for some free $R$-module $G$. Then there exists a short exact sequence of $R$-modules $0 \to Y \to A \to N \to 0$, where $\Gdim_R(A)<\infty$; see \cite[3.1]{Taksyz}.
\end{chunk}
\vspace{-0.1in}

Next is the proof of the second proposition:

\begin{proof} [Proof of Proposition \ref{p1}] Note that, in view of the short exact sequence $0\to M^{\oplus a} \to K \to \syz^{n} M^{\oplus b} \to 0$, it follows from Remark \ref{R1} that there are exact sequences
\begin{align}\tag{\ref{p1}.1}
0 \to \syz^{n+1}M^{\oplus b} \to M^{\oplus a}\oplus F \to K \to 0 
\end{align}
and
\begin{align} \tag{\ref{p1}.2}
0 \to \syz^{n+1}M^{\oplus a} \to G \oplus \syz^{n+1}K \to \syz^{2n+1}M^{\oplus b} \to 0,
\end{align}
where $F$ and $G$ are free $R$-modules.

By taking the direct sum of $a$ copies of the short exact sequence in (\ref{p1}.1) and the direct sum of $b$ copies of the short exact sequence in (\ref{p1}.2), we obtain the following short exact sequences:
\begin{align}\tag{\ref{p1}.3}
0 \to \syz^{n+1}M^{\oplus ab} \stackrel{\alpha}
     {\longrightarrow} M^{\oplus a^2}\oplus F^{\oplus a} \to K^{\oplus a} \to 0 
\end{align}
\begin{align} \tag{\ref{p1}.4}
0 \to \syz^{n+1}M^{\oplus ab} \stackrel{\beta}
     {\longrightarrow}  G^{\oplus b} \oplus \syz^{n+1}K^{\oplus b} \to \syz^{2n+1}M^{\oplus b^2} \to 0
\end{align}

Now we take the pushout of the maps $\alpha$ and $\beta$ from the exact sequences in (\ref{p1}.3) and (\ref{p1}.4), and obtain the following diagram with with exact rows and columns:
\[
\xymatrix{
& 0 \ar[d] & 0 \ar[d] & & \\
0 \ar[r] & \syz^{n+1}M^{\oplus ab} \ar[r]^{\alpha} \ar[d]^{\beta} & M^{\oplus a^{2}}\oplus F^{\oplus a} \ar[r] \ar[d] & K^{\oplus a} \ar[r] \ar@{=}[d] & 0\\
0 \ar[r] & G^{\oplus b}  \oplus \syz^{n+1}K^{\oplus b} \ar[r] \ar[d] & X \ar[r] \ar[d] & K^{\oplus a} \ar[r] & 0\\
& \syz^{2n+1}M^{\oplus b^2} \ar[d] \ar@{=}[r] & \syz^{2n+1}M^{\oplus b^2} \ar[d] & &\\
& 0 & 0 & &
}
\]
Now assume $\Gdim_R(K)<\infty$. Then the exact sequence in the middle row in the above diagram implies that $\Gdim_R(X)<\infty$. So we use \ref{Tak} for the exact sequence $0 \to M^{\oplus a^2}\oplus F^{\oplus a} \to X \to \syz^{2n+1} M^{\oplus b^2} \to 0$, and obtain a short exact sequence of the form
\begin{align} \tag{\ref{p1}.5}
0 \to M^{\oplus a^{2}} \to A \to \syz^{2n+1} M^{\oplus b^{2}} \to 0,
\end{align}
where $\Gdim_R(A)<\infty$. Therefore, setting $Y_1=A$, we establish the claim for the case where $i=1$.

Next assume $i\geq 2$. Then, by the induction hypothesis, there exists a short exact sequence of $R$-modules of the form
\begin{align} \tag{\ref{p1}.6}
0\to M^{\oplus a_{i-1}} \to Y_{i-1} \to \syz^{r_{i-1}} M^{\oplus b_{i-1}} \to 0,
\end{align}
where $\Gdim_R(Y_{i-1})<\infty$, $r_{i-1}=2^{i-1}(n+1)-1$, $a_{i-1}=a^{2^{i-1}}$ and $b_{i-1}=b^{2^{i-1}}$. Hence we can apply the previous process to the short exact sequence in (\ref{p1}.6) and obtain a short exact sequence of $R$-modules $0\to M^{\oplus a_i} \to Y_i \to \syz^{r_i} M^{\oplus b_i} \to 0$, where $\Gdim_R(Y_i)<\infty$, $r_i=2^i(n+1)-1$, $a_i=a^{2^{i}}$ and $b_i=b^{2^{i}}$, as required.
\end{proof}


To establish Proposition \ref{mainprop}, we prepare two lemmas:

\begin{lem} \label{l1} Let $R$ be a local ring, $M$ an $R$-module, and let $x \in \m$ be an element of $R$.
\begin{enumerate}[\rm(i)]
\item $x$ is a non zero-divisor on $R$, then $x$ is a non zero-divisor on $\Omega_R^i(M)$ for each $i \geq 1$.
\item If $x$ is a non zero-divisor on $R$ and $M$, then $\Omega_R^i(M)/x\Omega_R^i(M) \iso \Omega_{R/xR}^i(M/xM)$ for each $i\geq 0$.
\end{enumerate}
\end{lem}

\begin{proof} The first part is clear so we proceed by induction on $i$ to prove the second part.

If $i=0$, then $\Omega_R^0(M)/x\Omega_R^0(M)=M/xM=\Omega_{R/xR}^0(M/xM)$, so the claim follows. Next, given $i\geq 0$, we consider the following short exact sequence:
\begin{equation}\tag{\ref{l1}.1}
0 \rightarrow \Omega^{i+1} M \rightarrow R^{\oplus \beta_i} \rightarrow \Omega^iM \rightarrow 0
\end{equation}
Then, since $x$ is a non zero-divisor on $\Omega^iM$, we obtain, by tensoring (\ref{l1}.1) by $R/xR$, another short exact sequence in $\md(R/xR)$:
\begin{equation}\tag{\ref{l1}.2}
0 \rightarrow \Omega^{i+1}(M)/x\Omega^{i+1}M  \rightarrow R^{\oplus \beta_i}/xR^{\oplus \beta_i} \rightarrow \Omega^i(M)/x\Omega^i(M) \rightarrow 0.
\end{equation}
This implies that $\Omega^{i+1}(M)/x\Omega^{i+1}M \cong \Omega_{R/xR}\big(\Omega^i(M)/x\Omega^i(M)\big)$. As the induction hypothesis yields the isomorphism $\Omega_R^i(M)/x\Omega_R^i(M) \iso \Omega_{R/xR}^i(M/xM)$, we conclude that: 
\begin{equation}\notag{}
\Omega^{i+1}(M)/x\Omega^{i+1}M \cong \Omega_{R/xR}\big(\Omega_{R/xR}^i(M/xM)\big) \cong \Omega^{i+1}_{R/xR}\big(M/xM)\big). 
\end{equation}
This establishes the claim.
\end{proof}

\begin{lem} \label{rmk1} Let $R$ be a local ring and let  $0 \to K_{i-1}^{\oplus a_i} \to K_i \to \syz^{n_i}_R K_{i-1}^{\oplus b_i} \to 0$ be short exact sequences of $R$-modules for $i=1, \ldots r$, where $a_i, b_i, n_i$ are nonnegative integers and $r$ is a positive integer. If $x\in \fm$ is a non zero-divisor on $K_0$, then $x$ is a non zero-divisor on $K_i$ for each $i=0, \ldots, r$.
\end{lem}

\begin{proof} We proceed by induction on $i$. If $i=0$, then the claim is just the hypothesis. Hence we assume $i\geq 1$. Then, by the induction hypothesis, it follows that $x$ is a non-zero divisor on $K_{i-1}$. Thus, tensoring the given short exact sequences by $R/xR$, for each $i=1, \ldots, r$, we obtain an exact sequence of the form $\Tor_1^R(K_{i-1}^{\oplus a_i}, R/xR) \to \Tor_1^R(K_i,R/xR) \to \Tor_1^R(\syz^{n_i}_R K_{i-1}^{\oplus b_i}, R/xR)$. This yields $\Tor_1^R(K_i,R/xR) =0$ and hence shows that $x$ is a non zero-divisor on $K_i$, as required.
\end{proof}

We are now ready to give a proof of the third proposition:

\begin{proof} [Proof of Proposition \ref{mainprop}] There is nothing to prove if $\rGdim_R(M)=\infty$. Therefore we assume $\rGdim_R(M)=r<\infty$. Then, by definition, there exist short exact sequences of $R$-modules:
\begin{equation} \tag{\ref{mainprop}.1}
0 \to K_{i-1}^{\oplus a_i} \to K_i \to \syz^{n_i}_R K_{i-1}^{\oplus b_i} \to 0,
\end{equation}
where $K_0=M$ and $\Gdim_R(K_r)<\infty$.

Tensoring the short exact sequences in (\ref{mainprop}.1) by $R/xR$, we obtain the following exact sequences:
\begin{equation} \tag{\ref{mainprop}.2}
\Tor_1^R(\syz^{n_i}_R K_{i-1}^{\oplus b_i},\overline{R}) \to \big(\overline{K_{i-1}}\big)^{\oplus a_i} \to \overline{K_i} \to \overline{\big(\syz^{n_i}_R K_{i-1}^{\oplus b_i}\big)} \to 0,
\end{equation}
where $\overline{(-)}$ denotes $-\otimes_R R/xR$. 

Note that Lemma \ref{l1}(ii) implies $\overline{\big(\syz^{n_i}_R K_{i-1}^{\oplus b_i}\big)} \cong \syz^{n_i}_{\; \overline{R}} \big(\overline{K_{i-1}}\big)^{\oplus b_i}$. Note also that, by Remark \ref{rmk1}, $x$ is a non zero-divisor on each $K_i$ and hence $\Tor_1^R(\syz^{n_i}_R K_{i-1}^{\oplus b_i},\overline{R})$ vanishes for each $i=1, \ldots r$. Therefore, for each $i=1, \ldots r$, (\ref{mainprop}.2) yields the following exact sequence of $\overline{R}$-modules:
\begin{equation} \tag{\ref{mainprop}.3}
0 \to \big(\overline{K_{i-1}}\big)^{\oplus a_i} \to \overline{K_i} \to \syz^{n_i}_{\; \overline{R}} \big(\overline{K_{i-1}}\big)^{\oplus b_i} \to 0.
\end{equation}
Note that $\Gdim_{\overline{R}}(\overline{K_r})=\Gdim_R(K_r)<\infty$ \cite[1.4.5]{Gdimbook}. Therefore, we conclude by using the short exact sequences in (\ref{mainprop}.3) that $\rGdim_{\overline{R}}(M/xM)\leq r$.
\end{proof}

\section*{Acknowledgements}
The authors are grateful to Hiroki Matsui for his help on the proof of Theorem \ref{mainthm}(i), and to Arash Sadeghi for explaining them the argument given in \ref{Holm's}. The authors also thank Henrik Holm and Ryo Takahashi for helpful comments during the preparation of the manuscript.


\begin{thebibliography}{10}

\bibitem{Tome1}
\emph{Anneaux de {G}orenstein, et torsion en alg\`ebre commutative},
  S\'eminaire d'Alg\`ebre Commutative dirig\'e par Pierre Samuel, 1966/67.
  Texte r\'edig\'e, d'apr\`es des expos\'es de Maurice Auslander, Marquerite
  Mangeney, Christian Peskine et Lucien Szpiro. \'Ecole Normale Sup\'erieure de
  Jeunes Filles, Secr\'etariat math\'ematique, Paris, 1967.

\bibitem{Alp}
Jonathan~Lazare Alperin, \emph{Periodicity in groups}, Illinois J. Math.
  \textbf{21} (1977), no.~4, 776--783.

\bibitem{CA}
Tokuji Araya and Olgur Celikbas, \emph{Reducing invariants and total
  reflexivity}, Illinois J. Math. \textbf{64} (2020), no.~2, 169--184.

\bibitem{AT}
Tokuji Araya and Ryo Takahashi, \emph{On reducing homological dimensions over
  noetherian rings}, Preprint. posted at arXiv:2010.10765 (2020).

\bibitem{AuBr}
Maurice Auslander and Mark Bridger, \emph{Stable module theory}, Mem. Amer.
  Math. Soc., No. 94, American Mathematical Society, Providence, R.I., 1969.

\bibitem{AuBu}
Maurice Auslander and Ragnar-O. Buchweitz, \emph{The homological theory of
  maximal {C}ohen-{M}acaulay approximations}, M\'em. Soc. Math. France (N.S.)
  (1989), no.~38, 5--37.

\bibitem{Av1}
Luchezar~L. Avramov, \emph{Modules of finite virtual projective dimension},
  Invent. Math. \textbf{96} (1989), no.~1, 71--101.

\bibitem{Be}
Petter~Andreas Bergh, \emph{Modules with reducible complexity}, J. Algebra
  \textbf{310} (2007), 132--147.

\bibitem{Berghredcx2}
\bysame, \emph{Modules with reducible complexity {II}}, Comm. Algebra
  \textbf{37} (2009), no.~6, 1908--1920.

\bibitem{Gdimbook}
Lars~Winther Christensen, \emph{Gorenstein dimensions}, Lecture Notes in
  Mathematics, vol. 1747, Springer-Verlag, Berlin, 2000.

\bibitem{DaoTak}
Hailong Dao and Ryo Takahashi, \emph{Classification of resolving subcategories
  and grade consistent functions}, Int. Math. Res. Not. (2015), no.~1,
  119--149.

\bibitem{Foxby77}
Hans-Bj{\o}rn Foxby, \emph{Isomorphisms between complexes with applications to
  the homological theory of modules}, Math. Scand. \textbf{40} (1977), 5--19.

\bibitem{Holm}
Henrik Holm, \emph{Rings with finite {G}orenstein injective dimension}, Proc.
  Amer. Math. Soc. \textbf{132} (2004), no.~5, 1279--1283.

\bibitem{Ischebeck}
Friedrich Ischebeck, \emph{Eine dualit{\"a}t zwischen den funktoren {E}xt und
  {T}or}, J. Algebra \textbf{11} (1969), 510--531.

\bibitem{LV}
Gerson Levin and Wolmer~V. Vasconcelos, \emph{Homological dimensions and
  {M}acaulay rings}, Pacific J. Math. \textbf{25} (1968), 315--323.

\bibitem{Taksyz}
Ryo Takahashi, \emph{Syzygy modules with semidualizing or {G}-projective
  summands}, J. Algebra \textbf{295} (2006), no.~1, 179--194.

\bibitem{Yo}
Yuji Yoshino, \emph{Cohen-{M}acaulay modules over {C}ohen-{M}acaulay rings},
  London Mathematical Society Lecture Note Series, vol. 146, Cambridge
  University Press, Cambridge, 1990.

\end{thebibliography}
\end{document}